\documentclass[12pt, a4paper, twoside]{amsart}

\usepackage[T1]{fontenc}
\usepackage[applemac]{inputenc}
\usepackage{amsmath}
\usepackage{amsthm}
\usepackage{amsfonts}
\usepackage{bbm}
\usepackage[UKenglish]{babel}
\usepackage{enumerate}
\usepackage{bm}
\usepackage{ae}
\usepackage{amssymb}

\theoremstyle{definition}

\theoremstyle{remark}

\theoremstyle{plain}
\newtheorem{thm}{Theorem}
\newtheorem{lem}[thm]{Lemma}

\makeatletter
\newcommand*{\house}[1]{%
	\mathord{%
		\mathpalette\@house{#1}%
	}%
}
\newcommand*{\@house}[2]{%
	\dimen@=\fontdimen8 %
	\ifx#1\scriptscriptstyle\scriptscriptfont
	\else\ifx#1\scriptstyle\scriptfont
	\else\textfont\fi\fi
	3 %
	\sbox0{%
		$#1%
		\vrule width\dimen@\relax
		\overline{%
			\kern2\dimen@
			\begingroup 
			#2%
			\endgroup
			\kern2\dimen@
		}%
		\vrule width\dimen@\relax
		\mathsurround=1.5\dimen@ 
		$%
	}%
	\ht0=\dimexpr\ht0-\dimen@\relax
	\dp0=\dimexpr\dp0+2\dimen@\relax
	\vbox{%
		\kern\dimen@ 
		\copy0 %
	}%
}
\makeatother

\usepackage[normalem]{ulem}
\usepackage{xcolor}

\definecolor{darkgreen}{rgb}{0,0.7,0}

\begin{document}
	
	\title{The $p$-adic Duffin--Schaeffer conjecture}
	
	\author{SIMON KRISTENSEN}
	
	\address{S. Kristensen, Department of Mathematics, Aarhus University, Ny Munkegade 118,
		DK-8000 Aarhus C, Denmark}
	
	\email{sik@math.au.dk}

	\author{MATHIAS L\O{}KKEGAARD LAURSEN}
	
	\address{M. L. Laursen, Department of Mathematics, Aarhus University, Ny Munkegade 118,
		DK-8000 Aarhus C, Denmark}
	
	\email{mll@math.au.dk}

	\begin{abstract}
		We prove Haynes' version of the Duffin--Schaeffer conjecture for the $p$-adic numbers. In addition, we prove several results about an associated related but false conjecture, related to $p$-adic approximation in the spirit of Jarn\'ik and Lutz.
	\end{abstract}
	
	\dedicatory{For Asmus Schmidt -- in memoriam}
	\thanks{Research supported by the Independent Research Fund Denmark (Grant ref. 1026-00081B)}
	
	\maketitle
	
	\section{Introduction}
	
	A classical set to study in the metric theory of Diophantine approximation is the set
	\begin{equation}
		\label{eq:main}
		\mathcal{A}(\psi) = \left\{ x \in [0,1] : \left \vert x - \frac{a}{n} \right\vert < \frac{\psi(n)}{n} \text{ for infinitely many } \frac{a}{n} \in \mathbb{Q} \right\},
	\end{equation}
	where $\psi: \mathbb{N} \rightarrow [0,\infty)$ is some function and the rationals $a/n$ are assumed to be on lowest terms. The Duffin--Schaeffer conjecture \cite{MR4859} states that the Lebesgue measure of the set $\mathcal{A}(\psi)$ is 
	\begin{equation}
		\label{eq:DS}
		\vert \mathcal{A}(\psi) \vert = 
		\begin{cases} 
			0 & \text{if } \sum_{n=1}^\infty \frac{\psi(n) \phi(n)}{n} < \infty, \\
			1 & \text{if } \sum_{n=1}^\infty \frac{\psi(n) \phi(n)}{n} =\infty,
		\end{cases}
	\end{equation}
	where $\phi$ denotes the Euler totient function. The convergence case is a simple application of the Borel--Cantelli lemma, and the divergence case is the main problem in the conjecture. The conjecture was recently settled in the affirmative by Koukoulopoulos and Maynard \cite{MR4125453}.

	The main novelty of the Duffin--Schaeffer conjecture is that it removes a monotonicity condition from a celebrated theorem of Khintchine \cite{MR1512207}, which states that if $\psi$ is monotonic, one need only consider the series $\sum_{n=1}^\infty \psi(n)$ to obtain the required conclusion. In fact, Khintchine studied the set without the requirement that the approximating rationals are reduced, though this is not important in the case when $\psi$ is monotonic. It was observed by Duffin and Schaeffer that the monotonicity condition cannot be removed from Khintchine's theorem. They suggested that one could modify the divergence condition to the one above in order to get a result without the monotonicity condition. 
	
	Removing the coprimality condition on $a$ and $n$ results in a substantially different statement when $\psi$ is not decreasing. This led Catlin \cite{MR417098} to formulate a conjecture in this case, which was also settled in the affirmative by Koukoulopoulos and Maynard \cite{MR4125453}.
	
	In \cite{MR2576282}, Haynes gave a correspondence between a $p$-adic version of the conjecture and the real variable version. We will state a precise version of his result later, but briefly it states that if the Duffin--Schaeffer conjecture can be proved by a certain method, a $p$-adic analogue would follow. Conversely, if one could prove the $p$-adic by the same method, an almost complete Duffin--Schaeffer conjecture would follow. 
	
	Of course, in view of \cite{MR4125453} the second implication is no longer relevant. The converse implication, however, appears not to have been explored. It is the purpose of this note to show how the proof of the $p$-adic Duffin--Schaeffer conjecture can be deduced from the work of Haynes together with a mild strengthening of the methods of Koukoulopoulos and Maynard.

	Throughout, we will implicitly use the letter $p$ for a prime number. We will denote by $\nu_p$ the $p$-adic valuation, $\vert \cdot \vert_p$ the usual $p$-adic absolute value, by $\mathbb{Z}_p$ the $p$-adic integers, and by $\mu_p$ the Haar measure on the $p$-adic numbers, normalised so that $\mu_p(\mathbb{Z}_p)=1$. We will denote the greatest common divisor of $n,m \in \mathbb{N}$ by $\gcd(m,n)$. We will also use the Vinogradov notation, so that for two real quantities, $f$ and $g$, we will say that $f \ll g$ if there exists a constant $C > 0$ such that $f \le Cg$. In Landau's $O$-notation, which will also be used, this amounts to saying that $f = O(g)$.
	
	\section{Statement of main results}
	
	Let $p$ be a prime, let $\psi: \mathbb{N} \rightarrow \mathbb{R}_{\ge 0}$ be a function, and define for each natural number $n \in \mathbb{N}$ the set
	$$
	\mathcal{A}_n^p(\psi) = \bigcup_{\substack{-n \le a \le n \\ \gcd(a,n) =1}} \left\{x \in \mathbb{Z}_p : \left\vert x - \frac{a}{n} \right\vert_p \leq \frac{\psi(n)}{n} \right\}.
	$$
	Also, let
	\begin{alignat*}{2}
		\mathcal{A}^p(\psi) &= \limsup \mathcal{A}_n^p(\psi) \\
		&= \left\{x \in \mathbb{Z}_p : x \in \mathcal{A}_n^p \text{ for infinitely many } n\in \mathbb{N}\right\}.
	\end{alignat*}

	The $p$-adic Duffin--Schaeffer conjecture as stated by Haynes \cite{MR2576282} is the following Theorem.
	
	\begin{thm}
		\label{thm:main1}
		For any prime $p$ and any function $\psi: \mathbb{N} \rightarrow \mathbb{R}_{\ge 0}$,
		\begin{equation*}
			\mu_p(\mathcal{A}^p(\psi)) = \begin{cases}
				0 & \text{if } \sum_{n=1}^\infty \mu_p(\mathcal{A}_n^p) < \infty, \\
				1 & \text{if } \sum_{n=1}^\infty \mu_p(\mathcal{A}_n^p) = \infty.
			\end{cases}
		\end{equation*} 
	\end{thm}
	
	Note that some restriction on the numerators in the definition of the sets $\mathcal{A}_n^p$ is certainly required. Indeed, with unrestricted numerators, fix an $x \in \mathbb{Z}_p$ and a denominator $n \in \mathbb{N}$ and express the $p$-adic integer $nx$ as a series
	$$
	nx = \sum_{i=0}^\infty \gamma_i p^i.
	$$
	Then, for a natural number $N \in \mathbb{N}$, on letting 
	$$
	a_N = \sum_{i=0}^N \gamma_i p^i,
	$$
	we would have
	$$
	\left \vert x - \frac{a_N}{n} \right\vert_p = \vert n \vert_p^{-1} \left\vert \sum_{i=N+1}^\infty \gamma_i p^i\right\vert_p \le \vert n \vert_p^{-1} p^{-N-1}.
	$$
	This can be made arbitrarily small by increasing $N$, so that without the restriction on the numerators, the set would always have full Haar measure, unless of course $\psi(n) = 0$ for all but finitely many values of $n$.
	
	The statement of Theorem \ref{thm:main1} looks a little curious in comparison with the real Duffin--Schaeffer conjecture. Most notably, the Euler totient function is absent from the series governing the measure of the set $\mathcal{A}^p$. However, it is implicitly present in the definition of the sets $\mathcal{A}_n^p$. It can be brought back in the statement by giving equal weight to denominators and numerators, thus changing the conjecture slightly. We define sets
	\begin{alignat*}{2}
		\mathcal{B}_n^p(\psi) = &\bigcup_{\substack{-n \le a \le n \\ \gcd(a,n) =1}} \left\{x \in \mathbb{Z}_p : \left\vert x - \frac{a}{n} \right\vert_p \le \frac{\psi(n)}{n} \right\} \\
		&\cup \bigcup_{\substack{-n \le a \le n \\ \gcd(a,n) =1}} \left\{x \in \mathbb{Z}_p : \left\vert x - \frac{n}{a} \right\vert_p \le \frac{\psi(n)}{n} \right\},
	\end{alignat*}
	and as before let 
	$$
	\mathcal{B}^p(\psi) = \limsup \mathcal{B}_n^p(\psi).
	$$
	This setup of the problem is closer to the original setup for $p$-adic Diophantine approximation, as it was introduced by Jarn\'ik \cite{MR15092} and extensively developed by Lutz \cite{MR0069224}. We will show the following.
	
	\begin{thm}
		\label{thm:main2}
		For any prime $p$ and any function $\psi: \mathbb{N} \rightarrow \mathbb{R}_{\ge 0}$ with $\psi(n)=0$ whenever $p \vert n$,
		\begin{equation*}
			\mu_p(\mathcal{A}^p(\psi)) = \mu_p(\mathcal{B}^p(\psi)) = \begin{cases}
				0 & \text{if } \sum_{n=1}^\infty \frac{\phi(n)\psi(n)}{n} < \infty, \\
				1 & \text{if } \sum_{n=1}^\infty \frac{\phi(n)\psi(n)}{n} = \infty.
			\end{cases}
		\end{equation*} 
	\end{thm}
	
	In this case, the restriction on $\psi$ is needed, since otherwise the measure of the set need not satisfy a zero--one law. This was observed by Haynes \cite{MR2576282} by explicitly choosing a function such that the measure of the resulting set is $p^{-1}$. A simple modification of his example easily gives rise to functions such that $\mu_p(\mathcal{B}^p(\psi)) = p^{-k}$ for arbitrary values of $k$. It is tempting to believe that these are the only possible values for the Haar measure of $\mathcal{B}^p(\psi)$, but this is not the case. Indeed, we have the following.
	
	\begin{thm}
		\label{thm:main3}
		With no restrictions on the function $\psi: \mathbb{N} \rightarrow \mathbb{R}_{\ge 0}$, the Haar measure of $\mathcal{B}^p(\psi)$ can attain an uncountable number of values, with $1$ being the only possible value above $1/p$. If $p = 2$, the Haar measure of $\mathcal{B}^2(\psi)$ may attain any value in $[0, \frac{1}{2}]\cup\{1\}$.
	\end{thm}
	
	\section{Auxillary results}
	
	In this section, we cite the results from \cite{MR2576282}, \cite{MR4125453} and \cite{MR1099767}, which we will need in order to prove our main results. The key property from \cite{MR2576282} needed in the property of \emph{quasi-independence on average for $\psi$ in $\mathbb{R}$}. We define sets of real numbers corresponding to our sets $\mathcal{A}^p_n(\psi)$,
	$$
	\mathcal{A}^\infty_n(\psi) = \bigcup_{\substack{a=1 \\ (a,n) = 1}}^n \left[\frac{a}{n} - \psi(n), \frac{a}{n} + \psi(n)\right].
	$$
	Note that we explicitly remove negative numerators here, as the real analogue of $\mathbb{Z}_p$ is the unit interval $[0,1]$. We will say that the function $\psi$ satisfies the property of quasi-independence on average in $\mathbb{R}$, $(\text{QIA}^\infty,\psi)$ for short, if
	\begin{equation}
		\label{eq:QIA}
		\limsup_{N \rightarrow \infty} \left(\sum_{n \le N} \lambda(\mathcal{A}^\infty_n(\psi))\right)^2 \left(\sum_{n,m \le N} \lambda(\mathcal{A}^\infty_n(\psi) \cap \mathcal{A}^\infty_m(\psi))\right)^{-1} > 0.
	\end{equation}
	Here and elsewhere, $\lambda$ denotes the Lebesgue measure.
	
	The result of Haynes, which we will require and whose statement has been slightly modified here, is the following.
	
	\begin{thm}[Theorem 2 of \cite{MR2576282}]
		\label{thm:Haynes}
		If $(\text{QIA}^\infty,\psi)$ is satisfied for any functions $\psi: \mathbb{N} \rightarrow [0,\frac12]$, Theorem \ref{thm:main1} is true.
	\end{thm}
	
	Thus, to prove Theorem \ref{thm:main1} we need only verify $(\text{QIA}^\infty,\psi)$ for the required family of functions. For this, we will need results from \cite{MR1099767}. Before stating the required results, let us define some quantities. We fix a function $\psi: \mathbb{N} \rightarrow [0,\infty)$ and define for $m,n \in \mathbb{N}$
	$$
	M(m,n) = \max\{m\psi(n), n \psi(m)\}.
	$$
	Furthermore, for $a,b\in \mathbb{N}$ and $t \ge 1$, we let
	$$
	L_t(a,b) = \sum_{\substack{p \vert ab/\gcd(a,b)^2 \\ p \ge t}} \frac1p.
	$$
	With this notation, Koukoulopoulos and Maynard proved the following result.
	
	\begin{lem}[Proposition 5.4 of \cite{MR4125453}]
		\label{lem:KM}
		Let $\psi: \mathbb{N} \rightarrow [0,\frac12]$ be a function with $\sum_{n=1}^\infty \phi(n)\psi(n)/n = \infty$, and let $1 \le X \le Y$ be such that $1 \le \sum_{X \le n \le Y}^\infty \phi(n)\psi(n)/n \le 2$. Finally, let $t \ge 1$ and put
		$$
		\mathcal{E}_t = \left\{(v,w) \in (\mathbb{Z}\cap[X,Y])^2 : \gcd(v,w) \ge \frac{M(v,w)}{t}, \; L_t(v,w) \ge 10\right\}.
		$$
		Then,
		$$
		\sum_{(v,w) \in \mathcal{E}_t} \frac{\phi(v)\psi(v)}{v} \; \frac{\phi(w)\psi(w)}{w} \ll \frac{1}{t}.
		$$
	\end{lem}
	
	We will need the following Lemma of Pollington and Vaughan.
	
	\begin{lem}[From \S 3 of \cite{MR1099767}]
		\label{lem:PV}
		Let $\psi: \mathbb{N} \rightarrow [0,\frac12]$. For $m\neq n$,
		$$
		\frac{\lambda(\mathcal{A}^\infty_n(\psi) \cap \mathcal{A}^\infty_m(\psi))}{\lambda(\mathcal{A}^\infty_n(\psi))\lambda(\mathcal{A}^\infty_m(\psi))} \ll \mathbbm{1}_{M(m,n) \ge \gcd(m,n)} \prod_{\substack{p \vert mn/\gcd(m,n)^2 \\ p > M(m,n)/\gcd(m,n)}} \left(1+ \frac{1}{p}\right).
		$$
	\end{lem}
	
	Finally, we will need the classical estimate of Mertens \cite{MR1579612} that for $x \ge 2$
	\begin{equation}
		\label{eq:mertens}
		\sum_{p \le x} \frac{1}{p} = \log \log x + b + O\left(\frac{1}{\log x}\right).
	\end{equation}

	\section{Proof of the main theorems}
	
	We start with Theorem \ref{thm:main1}. For this purpose, we will need a mild extension of Lemma \ref{lem:KM}.
	
	\begin{lem}
		\label{lem:lemma1}
		Let $\psi: \mathbb{N} \rightarrow [0,\frac12]$ with $\sum_{n=1}^\infty \phi(n)\psi(n)/n = \infty$, let $K \in \mathbb{N}$, and let $1 \le X \le Y$ be such that $K \le \sum_{X \le n \le Y}^\infty \phi(n)\psi(n)/n \le 2K$. Finally, let $t \ge 1$ and put
		$$
		\mathcal{E}^K_t = \left\{(v,w) \in (\mathbb{Z}\cap[X,Y])^2 : \gcd(v,w) \ge \frac{M(v,w)}{t}, \; L_t(v,w) \ge 10\right\}.
		$$
		Then,
		$$
		\sum_{(v,w) \in \mathcal{E}_t} \frac{\phi(v)\psi(v)}{v} \; \frac{\phi(w)\psi(w)}{w} \ll \frac{K^2}{t}.
		$$
	\end{lem}
	
	\begin{proof}
		This follows readily on replacing $\psi$ by $\frac{1}K \psi$ in Lemma \ref{lem:KM} and noting that $\max\{n \frac1K \psi(m), m \frac1K \psi(n)\} = \frac1K M(m,n)$, which implies that the set $\mathcal{E}_t$ of Lemma \ref{lem:KM} with $\psi$ replaced by $\frac1K \psi$ contains the set $\mathcal{E}^K_t$.
	\end{proof}
	
	With this version of the Lemma, we can prove Theorem \ref{thm:main1}. As with the Lemma, the proof is a modification of the argument used for the Duffin--Schaeffer conjecture in \cite{MR4125453}.
	
	\begin{proof}[Proof of Theorem \ref{thm:main1}]
		First, a straightforward calculation shows that
		\begin{equation}
			\label{eq:measure}
			\lambda(\mathcal{A}^\infty_n(\psi)) = (2 - \mathbbm{1}_{n=1}) \frac{\phi(n)\psi(n)}{n}.
		\end{equation}
		Now, fix a $K \in \mathbb{N}$. Since $\sum_{n=1}^\infty \phi(n)\psi(n)/n = \infty$ with each summand $\le \frac12$, we can find an integer $N_K \in \mathbb{N}$ with
		$$
		\sum_{n=1}^{N_K} \frac{\phi(n)\psi(n)}{n} \in [K, K+1).
		$$
		We will prove that the $\limsup$ in $(\text{QIA}^\infty,\psi)$ along the sequence $(N_K)_{K=1}^\infty$ is positive. Clearly, this implies $(\text{QIA}^\infty,\psi)$, and hence by Theorem \ref{thm:Haynes} implies our Theorem \ref{thm:main1}.
		
		Clearly, by choice of $N_K$, 
		$$
		K^2 \le \left(\sum_{n=1}^{N_K} \frac{\phi(n)\psi(n)}{n}\right)^2.
		$$
		Hence, if we can prove that 
		\begin{equation}
			\label{eq:qia}
			\sum_{m,n \le N_K} \lambda(\mathcal{A}^\infty_n(\psi) \cap \mathcal{A}^\infty_m(\psi)) \ll K^2,
		\end{equation}
		we are done.
		
		We deal first with the diagonal terms, $n=m$. In this case, by \eqref{eq:measure}, 
		\begin{alignat*}{2}
			\sum_{m=n \le N_K} \lambda(\mathcal{A}^\infty_n(\psi) \cap \mathcal{A}^\infty_m(\psi)) &= \sum_{n=1}^{N_K} \lambda(\mathcal{A}^\infty_n(\psi))\\
			&\le 2 \sum_{n=1}^{N_K} \frac{\phi(n)\psi(n)}{n} \le 2K+2 \ll K^2,
		\end{alignat*}
		by choice of $N_K$. It thus suffices to prove that 
		$$
		\sum_{1 \le m < n \le N_K} \lambda(\mathcal{A}^\infty_n(\psi) \cap \mathcal{A}^\infty_m(\psi)) \ll K^2,
		$$
		by symmetry. Inserting \eqref{eq:measure} into the estimate of Lemma \ref{lem:PV}, it follows that \eqref{eq:qia} holds, provided
		\begin{equation}
			\label{eq:qia2}
			\sum_{\substack{1 \le m < n \le N_K \\ M(m,n) \ge \gcd(m,n)}}  \frac{\phi(n)\psi(n)}{n} \; \frac{\phi(m)\psi(m)}{m} \prod_{\substack{p \vert mn/\gcd(m,n)^2\\ p > M(m,n)/\gcd(m,n)}}\left(1+\frac1p\right)\ll K^2.
		\end{equation}
		
		We will first show that unless the latter product is large, we may ignore the contribution of the pair $(m,n)$. To see this, consider those pairs $(m,n)$ for which 
		\begin{equation}
			\label{eq:small}
			\prod_{\substack{p \vert mn/\gcd(m,n)^2\\ p > M(m,n)/\gcd(m,n)}}\left(1+\frac1p\right) \le e^{100}.
		\end{equation}
		For these pairs,
		\begin{alignat*}{2}
			\sum_{\substack{1 \le m < n \le N_K \\ M(m,n) \ge \gcd(m,n) \\ \text{\eqref{eq:small} holds}}} & \frac{\phi(n)\psi(n)}{n} \; \frac{\phi(m)\psi(m)}{m} \prod_{\substack{p \vert mn/\gcd(m,n)^2\\ p > M(m,n)/\gcd(m,n)}}\left(1+\frac1p\right) \\
			&\le e^{100} \sum_{\substack{1 \le m < n \le N_K \\ M(m,n) \ge \gcd(m,n)}}  \frac{\phi(n)\psi(n)}{n} \; \frac{\phi(m)\psi(m)}{m} \\
			&\le e^{100} \left(\sum_{n=1}^{N_K} \frac{\phi(n)\psi(n)}{n}\right)^2 \ll K^2.
		\end{alignat*}
		
		We may hence ignore these pairs, and consider only the sum over pairs $(m,n)$ for which
		\begin{equation}
			\label{eq:large}
			\begin{aligned}
				100 &\le \log\left(\prod_{\substack{p \vert mn/\gcd(m,n)^2\\ p > M(m,n)/\gcd(m,n)}}\left(1+\frac1p\right)\right) 
				\\&= \sum_{\substack{p \vert mn/\gcd(m,n)^2\\ p > M(m,n)/\gcd(m,n)}} \log\left(1+\frac1p\right).
			\end{aligned}
		\end{equation}
		Since $\log(1+x) \le x$, for $x > 0$, pairs satisfying \eqref{eq:large} must satisfy
		\begin{equation}
			\label{eq:large2}
			\begin{aligned}
			100 &\le \log\left(\prod_{\substack{p \vert mn/\gcd(m,n)^2\\ p > M(m,n)/\gcd(m,n)}}\left(1+\frac1p\right)\right)\le \sum_{\substack{p \vert mn/\gcd(m,n)^2\\ p > M(m,n)/\gcd(m,n)}} \frac{1}{p} \\&\le \sum_{p \vert mn/\gcd(m,n)^2} \frac1p.
			\end{aligned}
		\end{equation}
		
		Now, let $j =j(m,n)$ be the largest integer such that
		$$
		L_{\exp \exp(j)}(m,n) = \sum_{\substack{p \vert mn/\gcd(m,n)^2\\ p > \exp \exp(j)}} \frac{1}{p} \ge 10.
		$$
		By maximality,
		$$
		\sum_{\substack{p \vert mn/\gcd(m,n)^2\\ p > \exp \exp(j+1)}} \frac{1}{p} < 10.
		$$
		By Mertens' estimate \eqref{eq:mertens},
		$$
		\sum_{p > \exp \exp(j)} \frac{1}{p} = j +O(1),
		$$
		so that 
		$$
		\sum_{\exp \exp(j)\le p < \exp \exp(j+1)} \frac{1}{p} = O(1).
		$$
		Thus, 
		$$
		\sum_{\substack{p \vert mn/\gcd(m,n)^2\\ p > \exp \exp(j)}} \frac1p \le \sum_{\substack{p \vert mn/\gcd(m,n)^2\\ p > \exp \exp(j+1)}} \frac1p + \sum_{\exp \exp(j)\le p < \exp \exp(j+1)} \frac{1}{p} = O(1).
		$$
		For pairs $(m,n)$ satisfying \eqref{eq:large}, it follows on applying \eqref{eq:large2} that
		\begin{alignat*}{2}
			\prod_{\substack{p \vert mn/\gcd(m,n)^2\\ p > M(m,n)/\gcd(m,n)}} \left(1+\frac1p\right)  &\le \exp\left(\sum_{\substack{p \vert mn/\gcd(m,n)^2\\ p > M(m,n)/\gcd(m,n)}} \frac1p\right) \\
			&\ll \begin{cases}
				1 & \text{for } M(m,n)/\gcd(m,n) \ge \exp \exp(j), \\
				e^j & \text{otherwise.}
			\end{cases}
		\end{alignat*}
		
		By the above, the pairs with $M(m,n)/\gcd(m,n) \ge \exp \exp(j)$ contribute in total with $\ll K^2$ to the sum of \eqref{eq:qia2} by applying the same argument as when the products were bounded by $e^{100}$. Thus, these pairs may also be ignored, and we are left with the pairs $(m,n)$ for which $\exp \exp(j) \ge M(m,n)/\gcd(m,n)$. We order these by the size of $j= j(m,n)$ to see that these contribute at most
		\begin{alignat*}{2}
			\sum_{\substack{1 \le m < n \le N_K \\ M(m,n) \ge \gcd(m,n)\\ \exp \exp(j) \ge M(m,n)/\gcd(m,n)}}  &\frac{\phi(n)\psi(n)}{n} \; \frac{\phi(m)\psi(m)}{m} \prod_{\substack{p \vert mn/\gcd(m,n)^2\\ p > M(m,n)/\gcd(m,n)}}\left(1+\frac1p\right) \\
			& \ll \sum_{t=0}^\infty e^t \sum_{\substack{1 \le m < n \le N_K \\  \exp \exp(t) \ge M(m,n)/\gcd(m,n) \\ L_{\exp \exp(t)}(m,n) \ge 10}}  \frac{\phi(n)\psi(n)}{n} \; \frac{\phi(m)\psi(m)}{m} \\
			&= \sum_{t=0}^\infty e^t \sum_{(m,n) \in \mathcal{E}_{\exp \exp(t)}} \frac{\phi(n)\psi(n)}{n} \; \frac{\phi(m)\psi(m)}{m},
		\end{alignat*}
		where $\mathcal{E}_t$ is as in Lemma \ref{lem:lemma1} with $X=1$ and $Y=N_K$. Thus, by that lemma,
		\begin{alignat*}{2}
			\sum_{\substack{1 \le m < n \le N_K \\ M(m,n) \ge (m,n)\\ \exp \exp(j) \ge M(m,n)/\gcd(m,n)}}&  \frac{\phi(n)\psi(n)}{n} \; \frac{\phi(m)\psi(m)}{m} \prod_{\substack{p \vert mn/\gcd(m,n)^2\\ p > M(m,n)/\gcd(m,n)}}\left(1+\frac1p\right) \\
			&\ll \sum_{t=0}^\infty e^t \frac{K^2}{\exp \exp(t)} \ll K^2.
		\end{alignat*}
		This completes the proof of Theorem \ref{thm:main1}.
	\end{proof}
	
	We now turn to the proof of Theorem \ref{thm:main2}. 
	\begin{proof}[Proof of Theorem \ref{thm:main2}]
		As usual, the convergence case is the easier case, so suppose that $\sum_{n=1}^\infty \phi(n)\psi(n)/n < \infty$. Clearly, $\mathcal{A}^p_n(\psi) \subseteq \mathcal{B}^p_n(\psi)$, and since evidently $\mu_p(\mathcal{B}^p_n(\psi)) \le 4 \phi(n)\psi(n)/n$,
		$$
		\sum_{n=1}^\infty \mu_p(\mathcal{A}^p_n(\psi)) \le \sum_{n=1}^\infty \mu_p(\mathcal{B}^p_n(\psi)) \le 4 \sum_{n=1}^\infty \frac{\phi(n)\psi(n)}{n} < \infty,
		$$
		and the convergence part of Theorem \ref{thm:main2} follows immediately by the Borel--Cantelli Lemma.
		
		For the divergence case, suppose that $\sum_{n=1}^\infty \phi(n)\psi(n)/n = \infty$. Since clearly $\mathcal{A}^p(\psi) \subseteq \mathcal{B}^p(\psi)$, it suffices to prove that $\mu_p(\mathcal{A}^p(\psi)) = 1$. By Theorem \ref{thm:main1}, this will follow if $\sum_{n=1}^\infty \mu_p(\mathcal{A}^p_n(\psi)) = \infty$. We will prove the latter.
		
		Note that $\mathcal{A}_n^p(\psi)=\mathcal{A}_n^p(p^{-k_n})$ when $\psi(n)\neq 0$ and $k_n$ is the smallest non-negative integer such that $\psi(n)\geq p^{-k_n}$.
			We may therefore assume, without loss of generality, that $\psi(n)$ is always either $0$ or a non-negative power of $p$.
		From the first part of \cite[Lemma 4]{MR2576282}, translated into our notation, we find that if $\psi(n) < n/(12\phi(n))$, 
		$$
		\mu_p(\mathcal{A}^p_n(\psi)) \ge \frac{\phi(n)\psi(n)}{n}.
		$$
		Note that the lower bound assumed on $\psi$ in \cite[Lemma 4]{MR2576282} plays no role in this estimate. Hence,
		\begin{alignat*}{2}
			\sum_{n=1}^\infty \mu_p&(\mathcal{A}^p_n(\psi)) = \sum_{\psi(n) \ge n/(12\phi(n))} \mu_p(\mathcal{A}^p_n(\psi)) +\sum_{\psi(n) < n/(12\phi(n))}\mu_p(\mathcal{A}^p_n(\psi)) \\
			&\ge \sum_{\psi(n) \ge n/(12\phi(n))} \frac{\phi(n) n/(12\phi(n))}{n} + \sum_{\psi(n) < n/(12\phi(n))} \frac{\phi(n)\psi(n)}{n} \\
			&= \sum_{\psi(n) \ge n/(12\phi(n))} \frac{1}{12} + \sum_{\psi(n) < n/(12\phi(n))} \frac{\phi(n)\psi(n)}{n}.
		\end{alignat*}
		If $\psi(n) \ge n/(12\phi(n)) \ge 1/12$ for infinitely many values of $n$, the first series diverges, and hence the series $\sum_{n=1}^\infty \mu_p(\mathcal{A}^p_n(\psi))$ also diverges. Otherwise, since we have assumed that $\sum_{n=1}^\infty \phi(n)\psi(n)/n = \infty$, the latter series must diverge. Hence, in either case, $\sum_{n=1}^\infty \mu_p(\mathcal{A}^p_n(\psi)) = \infty$, and the theorem follows.
	\end{proof}
	
	We now proceed with the proof of Theorem \ref{thm:main3}. Initially, we consider the postulated gap in the possible measures attained between $1/p$ and $1$. This is the content of the following Lemma.
	
	\begin{lem}
		\label{lem:zero-one}
		Let $\psi: \mathbb{N} \rightarrow \mathbb{R}_{\ge 0}$, and suppose that $\mu_p(\mathcal{B}^p(\psi)) > 1/p$. Then, $\mu_p(\mathcal{B}^p(\psi)) = 1$.
	\end{lem}
	
	\begin{proof}
		Since $\mathcal{A}^p(\psi) \subseteq \mathcal{B}^p(\psi)$, we are done if $\mu_p(\mathcal{A}^p(\psi)) = 1$, so suppose this is not the case. In this case, by \cite[Lemma 1]{MR2576282}, $\mu_p(\mathcal{A}^p(\psi)) = 0$. Now, applying Theorem \ref{thm:main2} with the function $\psi \cdot \mathbbm{1}_{p \nmid n}$, we find that $\mu_p(\mathcal{B}^p(\psi \cdot \mathbbm{1}_{p \nmid n})) = 0$. Hence, the bulk of the mass of $\mathcal{B}^p(\psi)$ is carried by the $\mathcal{B}^p_n(\psi)$ for which $p \mid n$, i.e.,
		\[
		\mu_p(\limsup_{m \rightarrow \infty} \mathcal{B}^p_{mp}(\psi)) = \mu_p (\mathcal{B}^p(\psi)) > \frac{1}{p}.
		\]
		Consequently, since $\mu_p(p\mathbb{Z}_p) = 1/p$, we may pick an element $\alpha \in \limsup_{m \rightarrow \infty} \mathcal{B}^p_{mp}(\psi) \setminus p \mathbb{Z}_p$.
		
		With $\alpha$ fixed, pick an increasing sequence of positive integers, $m_1 < m_2 < \dots$ with $\alpha \in \mathcal{B}^p_{pm_j}(\psi)$ for all $j \in \mathbb{N}$. Then, for each $j \in \mathbb{N}$, there is an $a \in  \mathbb{Z}$ with $\vert a \vert \le p m_j$ and $\gcd(a, pm_j) = 1$, such that
		\[
		\min\left\{ \left\vert \alpha - \frac{a}{pm_j}\right\vert_p, \left\vert \alpha - \frac{pm_j}{a}\right\vert_p \right\} \le \frac{\psi(pm_j)}{p m_j}.
		\]
		Since $\vert pm_j/a \vert_p = \vert pm_j \vert_p < 1 = \vert \alpha \vert_p < \vert a/(pm_j) \vert_p$, we immediately find that 
		\[
		\min\left\{ \left\vert \alpha - \frac{a}{pm_j}\right\vert_p, \left\vert \alpha - \frac{pm_j}{a}\right\vert_p \right\} \ge \min\left\{ \left\vert \frac{a}{pm_j}\right\vert_p, \left\vert \alpha \right\vert_p \right\} = 1,
		\]
		whence $\psi(pm_j)/(p m_j) \ge 1$, so that $\mathcal{B}^p_{pm_j}(\psi) = \mathbb{Z}_p$ for any $j \in \mathbb{N}$. Hence, $\mathcal{B}^p(\psi) = \mathbb{Z}_p$, and in particular $\mu_p(\mathbb{Z}_p) = 1$.
	\end{proof}
	
	To complete the proof of Theorem \ref{thm:main3}, we provide two examples. The first will furnish us with a family of functions $\psi_k$ for which $\mu_p(\mathcal{B}^p(\psi_k)) = (p-1)p^{-k-1}$. The second will modify this example slightly to provide an uncountable number of possible values for $\mu_p(\mathcal{B}^p(\psi_k))$, which in the case for $p=2$ will coincide with the set $[0,\frac{1}{2}]$. Together with any function making the series of Theorem \ref{thm:main2} diverge and Lemma \ref{lem:zero-one}, this shows that the spectrum of values for  $\mu_2(\mathcal{B}^2(\psi))$ is $\{1\} \cup [0,\frac{1}{2}]$. 
	
	For the first example, let $k \in \mathbb{N}$, and let $\psi_k: \mathbb{N} \rightarrow \mathbb{R}_{\ge 0}$ be given by 
	\begin{equation}
		\label{eq:psi_k}
		\psi_k(n) = 
		\begin{cases}
			n/p^{k+1} & \text{whenever } \nu_p(n) = k, \\
			0 & \text{otherwise.}
		\end{cases}
	\end{equation}
	We claim that $\mathcal{B}^p(\psi_k) = p^k\mathbb{Z}_p \setminus p^{k+1}\mathbb{Z}_p$, so that $\mu_p(\mathcal{B}p(\psi_k)) = (p-1)p^{-k-1}$.
	
	To see this, note first that only the sets $\mathcal{B}^p_n(\psi_k)$ with $\nu_p(n) = k$ occur as non-empty sets in the limsup set $\mathcal{B}^p(\psi_k)$. Fix such an $n$, and consider the sets making up $\mathcal{B}^p_n(\psi_k)$. Since $\nu_p(n) = k$, $\vert a/n \vert_p = p^k$ for $\gcd(a,n) = 1$, so these fractions are too far away from the $p$-adic integers to contribute to the approximation. Therefore, the sets defined by this inequality must be empty, and hence,
	\begin{align*}
		\mathcal{B}^p_n(\psi_k) &= \bigcup_{\substack{\vert a \vert \le n \\ \gcd(a,n) = 1}} \left\{\alpha \in \mathbb{Z}_p : \left\vert \alpha - \frac{n}{a}\right\vert_p \le p^{-k-1}\right\} \\&
		= \bigcup_{\substack{\vert a \vert \le n \\ \gcd(a,n) = 1}} \frac{n}{a} + p^{k+1} \mathbb{Z}_p.
	\end{align*}
	Since $p \nmid a$, $\nu_p(n/a) = \nu_p(n) = k$, and since $\nu_p(\alpha) \ge k+1$ for any $\alpha \in p^{k+1}\mathbb{Z}_p$, we find that any element in $\mathcal{B}_n^p(\psi_k)$ must have $p$-adic valuation exactly equal to $k$, i.e.
	\begin{equation}
		\label{eq:counterex}
		\mathcal{B}_n^p(\psi_k) \subseteq p^k\mathbb{Z}_p \setminus p^{k+1} \mathbb{Z}_p.
	\end{equation}
	Since this is true for all $n$ with $\nu_p(n) = k$, and since $\mathcal{B}_n^p(\psi_k) = \emptyset$ if $\nu_p(n) \neq k$,

	For the converse inclusion, let $q > p$ be a prime, and consider $n = p^kq$. Evidently,
	\[
	\mathcal{B}^p_{n}(\psi_k) \supseteq \bigcup_{a=1}^{p-1} \frac{q}{a}p^k + p^{k+1}\mathbb{Z}_p,
	\]
	since $\gcd(a,q) = \gcd(a,p) = 1$ for $a=1,2,\dots, p-1$. Hence, $\vert a \vert_p = 1$, and since $\gcd(p,q)=1$, $\vert q \vert_p = 1$, so that $\vert q/a \vert_p = 1$, i.e. $q/a$ is a $p$-adic integer of norm $1$. Hence, for each $a$, we may write $q/a = b + p r_1$, where $r_1$ is a $p$-adic integer. Now, inversion modulo $p$ is a permutation of $\{1,2, \dots, p-1\}$, and so is subsequent multiplication by $q$, since $\gcd(p,q) = 1$. Absorbing the terms $p r_i$ into $p^{k+1}\mathbb{Z}_p$, it follows that
	\[
	\bigcup_{a=1}^{p-1} \frac{q}{a}p^k + p^{k+1}\mathbb{Z}_p = \bigcup_{b=1}^{p-1} bp^k + p^{k+1}\mathbb{Z}_p = p^k\mathbb{Z}_p \setminus p^{k+1} \mathbb{Z}_p.
	\]
	Since this is true for all infinitely many primes $q>p$, we thus have 
		\[
		\mathcal{B}^p(\psi_k) \supseteq p^k\mathbb{Z}_p \setminus p^{k+1} \mathbb{Z}_p.
		\]
	
	This provides us with a countable family of possible values for the measure. To get to an uncountable number and hence to the proof of Theorem \ref{thm:main3}, we let $x = (x_k)_{k=1}^\infty$ be a sequence with $x_k \in \{0,1\}$ for all $k$. Now, define 
	\begin{equation*}
		\psi^x(n) = 
		\begin{cases}
			n/p^{\nu_p(n) + 1} & \text{if $p \vert n$ and $x_{\nu_p(n)} = 1$}, \\
			0 & \text{otherwise.}
		\end{cases}
	\end{equation*}
	We calculate the measure of the sets $\mathcal{B}^p(\psi^x)$.
	
	Note first that 
	\[
	\psi^x(n) = \max_{\substack{k \in \mathbb{N} \\ x_k = 1}} \psi_k(n),
	\]
	where $\psi_k$ is given by \eqref{eq:psi_k}. Hence, if $x_k = 1$, $\mathcal{B}^p(\psi_k) \subseteq \mathcal{B}^p(\psi^x)$, so that
	\[
	\mathcal{B}^p(\psi^x) \supseteq \bigcup_{\substack{k \in \mathbb{N} \\ x_k = 1}} \mathcal{B}^p(\psi_k) = \bigcup_{\substack{k \in \mathbb{N} \\ x_k = 1}} p^k\mathbb{Z}_p \setminus p^{k+1} \mathbb{Z}_p.
	\]
	We now prove the converse inclusion. 
	
	Let $\alpha \in \mathbb{Z}_p$ with $\alpha \notin \bigcup_{\substack{k \in \mathbb{N} \\ x_k = 1}} p^k\mathbb{Z}_p \setminus p^{k+1} \mathbb{Z}_p$. Let $n \in \mathbb{N}$ with $\psi^x(n) > 0$. Then, 
	\[
	\mathcal{B}_n^p (\psi^x) = \mathcal{B}^p_n(\psi_{\nu_p(n)}) \subseteq p^{\nu_p(n)}\mathbb{Z}_p \setminus p^{\nu_p(n)+1}\mathbb{Z}_p,
	\]
	by \eqref{eq:counterex}. Since $\psi^x(n) > 0$ only if $x_{\nu_p(n)} = 1$, we immediately find that $\alpha \notin p^{\nu_p(n)}\mathbb{Z}_p \setminus p^{\nu_p(n)+1}\mathbb{Z}_p$ for any $n \in \mathbb{N}$, whence $\alpha \notin \mathcal{B}^p(\psi^x)$. Consequently, we have shown that 
	\[
	\mathcal{B}^p(\psi^x) = \bigcup_{\substack{k \in \mathbb{N} \\ x_k = 1}} p^k\mathbb{Z}_p \setminus p^{k+1} \mathbb{Z}_p.
	\]
	
	The sets on the right hand side are clearly disjoint, with each having Haar measure $(p-1)p^{-k-1}$, so we finally arrive at 
	\[
	\mu_p(\mathcal{B}^p(\psi^x)) = \bigcup_{\substack{k \in \mathbb{N} \\ x_k = 1}} \mu_p(p^k\mathbb{Z}_p \setminus p^{k+1} \mathbb{Z}_p) = \sum_{k=1}^\infty x_k (p-1) p^{-k-1}.
	\]
	These series clearly attain an uncountable number of values. Furthermore, if $p=2$, every real number  $y \in [0, 1/2]$ can be attained by choosing the $x_k$ to be the coefficients in a binary expansion of $y$. This completes the proof of Theorem \ref{thm:main3}. \qed
	
	The functions constructed above take their inspiration in a function constructed by Haynes in  \cite{MR2576282}. We will now explain the relation between the two constructions. We have constructed functions $\psi_k$ as building blocks satisfying $\mathcal{B}^p(\psi_k)=p^k\mathbb{Z}_p\setminus p^{k+1}\mathbb{Z}_p$ and
	\begin{align*}
		\mathcal{B}^p\left(\sum_{j=1}^{\infty}\psi_{k_j}\right) &= \bigcup p^{k}\mathbb{Z}_p\setminus p^{k+1}\mathbb{Z}_p
		= \bigcup \mathcal{B}^p (\psi_{k_j})
	\end{align*}
	where $k_j$ denotes an increasing sequence of positive integers.
	The construction of $\psi_k$ was initially inspired by a function $\psi'$ from \cite{MR2576282}, so that $\psi'(n)$ is $n/p$ when $p\mid n$ and $0$ when $p\nmid n$.
	For this function, Haynes shows that $\mathcal{B}^p(\psi')=\mathcal{B}^p_{n}(\psi')=p\mathbb{Z}_p$ for $p\mid n$. Of course, putting $\psi'_k(n)$ equal to $n/p^k$ when $p^k\mid n$ and equal to 0 when $p^k\nmid n$, one easily modifies Haynes' argument to see that $\mathcal{B}^p_{n}(\psi'_k)=p^{k}\mathbb{Z}_p$ for $p^k \mid n$ and $\mathcal{B}_n^p(\psi'_k)$ is finite and disjoint for different values of $n$ otherwise, so that $\mathcal{B}^p(\psi'_k)=p^{k}\mathbb{Z}_p$.  
	
	Our functions $\psi_k$ may be in this context be viewed as refinements of $\psi'$, constructed in such a way that they will remove the ball $p^{k+1}\mathbb{Z}_p$ from the ball $p^k \mathbb{Z}_p$ in the limsup set. This is accomplished by defining  $\psi_k$ to be $\psi'_k/p$ with its support restricted to the positive integers $n\in\mathbb{N}$ with $\nu_p(n)=k+1$ Indeed, comparing $\psi_k$ and $\psi'_{k+1}$, it is easy to see that
	\begin{align*}
		\mathcal{B}^p_{n}(\psi'_k/p) &= \begin{cases}
			\mathcal{B}^p_{n}(\psi_k) = p^k \mathbb{Z}_p\setminus p^{k+1}\mathbb{Z}_p	&\text{for $p^{k+2}\nmid n$}	\\
			\mathcal{B}^p_{n}(\psi'_{k+1}) = p^{k+1}\mathbb{Z}_p &\text{for $p^{k+2}\mid n$}.
		\end{cases}
	\end{align*}
	
	\section{Concluding remarks}
	
	The spectrum of values attainable by $\mu_p(\mathcal{B}^p(\psi))$ as $\psi$ varies remains somewhat elusive. We have shown that it is contained in the set $[0,1/p] \cup \{1\}$, and that if $p=2$, it is in fact equal to this set.
	
	When $p \neq 2$, less is known. We have seen that the spectrum of values is uncountable, and appealing for instance to the classical result of Hutchinson \cite{MR625600}, we easily find that the Hausdorff dimension of the spectrum is lower bounded by $\log 2/\log p$, just by considering the functions $\psi^x$ studied above.
	
	It appears natural to conjecture that the spectrum of values is equal to  $[0,1/p] \cup \{1\}$, or -- more modestly -- that the spectrum has positive Lebesgue measure. However, at present we do not know how to improve upon the estimates of the present paper. One might hope to find further possible values of the Haar measure of $\mathcal{B}^p(\psi)$ by doing further modifications to our functions $\psi^x$, but this approach seems futile as it can be shown that $\mathcal{B}_{p^k q}^p(\psi_k / p^l) = \mathcal{B}^p(\psi_k) = p^k\mathbb{Z}_p\setminus p^{k+1}\mathbb{Z}_p$ for infinitely many primes $q$ when $l\geq 0$. To see this, follow the arguments of the proof of Theorem \ref{thm:main3} where the primes $q$ are now restricted to be greater than $p^{l+1}$, and the inclusion $\mathcal{B}_{p^k q}\supseteq p^k\mathbb{Z}_p\setminus p^{k+1}\mathbb{Z}_p$
	is shown via
	\[
	\mathcal{B}_{p^k q}\supseteq \bigcup_{\substack{1 \le a <p^{l+1} \\ \gcd(a,p) =1}} \frac{q}{a} p^{k} + p^{k+l+1}\mathbb{Z}_p.
	\]
	The restriction to primes $q$ in the argument is mainly for the sake of simplicity so that we do not have to handle divisibility with primes less than $p^{l+1}$.

	\vspace{6pt}
		\paragraph{\textbf{Acknowledgements}} We thank the referee for helpful comments.
	
	\providecommand{\bysame}{\leavevmode\hbox to3em{\hrulefill}\thinspace}
	\providecommand{\MR}{\relax\ifhmode\unskip\space\fi MR }
	\providecommand{\MRhref}[2]{%
		\href{http://www.ams.org/mathscinet-getitem?mr=#1}{#2}
	}
	\providecommand{\href}[2]{#2}


\begin{thebibliography}{10}
		
		\bibitem{MR417098}
		P.~A. Catlin, \emph{Two problems in metric {D}iophantine approximation. {I}},
		J. Number Theory \textbf{8} (1976), no.~3, 282--288. 
		
		\bibitem{MR4859}
		R.~J. Duffin and A.~C. Schaeffer, \emph{Khintchine's problem in metric
			{D}iophantine approximation}, Duke Math. J. \textbf{8} (1941), 243--255.
		
		\bibitem{MR2576282}
		A.~K. Haynes, \emph{The metric theory of {$p$}-adic approximation}, Int.
		Math. Res. Not. IMRN (2010), no.~1, 18--52. 
		
		\bibitem{MR625600}
		J.~E. Hutchinson, \emph{Fractals and self-similarity}, Indiana Univ. Math. J.
		\textbf{30} (1981), no.~5, 713--747. 
		
		\bibitem{MR15092}
		V. Jarn\'{\i}k, \emph{Sur les approximations diophantiques des nombres
			{$p$}-adiques}, Rev. Ci. (Lima) \textbf{47} (1945), 489--505.
		
		\bibitem{MR1512207}
		A. Khintchine, \emph{Einige {S}\"{a}tze \"{u}ber {K}ettenbr\"{u}che, mit
			{A}nwendungen auf die {T}heorie der {D}iophantischen {A}pproximationen},
		Math. Ann. \textbf{92} (1924), no.~1-2, 115--125. 
		
		\bibitem{MR4125453}
		D. Koukoulopoulos and J. Maynard, \emph{On the {D}uffin-{S}chaeffer
			conjecture}, Ann. of Math. (2) \textbf{192} (2020), no.~1, 251--307.
		
		\bibitem{MR0069224}
		E. Lutz, \emph{Sur les approximations diophantiennes lin\'{e}aires
			{$P$}-adiques}, Actualit\'{e}s Scientifiques et Industrielles, No. 1224, Hermann \& Cie, Paris, 1955.
		
		\bibitem{MR1579612}
		F. Mertens, \emph{Ein {B}eitrag zur analytischen {Z}ahlentheorie}, J. Reine
		Angew. Math. \textbf{78} (1874), 46--62. 
		
		\bibitem{MR1099767}
		A.~D. Pollington and R.~C. Vaughan, \emph{The {$k$}-dimensional {D}uffin and
			{S}chaeffer conjecture}, Mathematika \textbf{37} (1990), no.~2, 190--200.
		
	\end{thebibliography}
\end{document}